\documentclass[12pt]{article}

\usepackage{comment}

\usepackage{amsmath, amsthm, amssymb}
\usepackage{enumerate}
\newtheorem{theorem}{}[section]
\newtheorem{lemma}[theorem]{}

\newtheoremstyle{styleclaim}{}{}{\itshape}{}{}{}{ }{(\thmnumber{#2})}

\theoremstyle{styleclaim}

\newcommand{\cref}[1]{(\ref{#1})}

\usepackage{tikz}

\usepackage{tkz-graph}
\begin{document}

\title{Simplicial vertices in graphs with no induced four-edge path or four-edge antipath, and the $H_6$-conjecture}

\date{\today}
\author{Maria Chudnovsky\thanks{Columbia University, New York, NY 10027, USA. E-mail: mchudnov@columbia.edu. Partially supported by NSF grants IIS-1117631 and DMS-1001091.} \and Peter Maceli\thanks{Columbia University, New York, NY 10027, USA. E-mail: plm2109@columbia.edu.}}

\maketitle

\begin{abstract}
Let $\mathcal{G}$ be the class of all graphs with no induced four-edge path or four-edge antipath. Hayward and Nastos \cite{MS} conjectured that every prime graph in $\mathcal{G}$ not isomorphic to the cycle of length five is either a split graph or contains a certain useful arrangement of simplicial and antisimplicial vertices. In this paper we give a counterexample to their conjecture, and prove a slightly weaker version. Additionally, applying a result of the first author and Seymour \cite{grow} we give a short proof of Fouquet's result \cite{C5} on the structure of the subclass of bull-free graphs contained in $\mathcal{G}$.

\end{abstract}

\section{Introduction}

All graphs in this paper are finite and simple. Let $G$ be a graph. The \textit{complement} $\overline{G}$ of $G$ is the graph with vertex set $V(G)$, such that two vertices are adjacent in $G$ if and only if they are non-adjacent in $\overline{G}$. For a subset $X$ of $V(G)$, we denote by $G[X]$ \textit{the subgraph of $G$ induced by $X$}, that is, the subgraph of $G$ with vertex set $X$ such that two vertices are adjacent in $G[X]$ if and only if they are adjacent in $G$. Let $H$ be a graph. If $G$ has no induced subgraph isomorphic to $H$, then we say that $G$ is \textit{$H$-free}. If $G$ is not $H$-free, \textit{$G$ contains $H$}, and \textit{a copy of $H$ in $G$} is an induced subgraph of $G$ isomorphic to $H$. For a family $\mathcal{F}$ of graphs, we say that $G$ is \textit{$\mathcal{F}$-free} if $G$ is $F$-free for every $F\in \mathcal{F}$. 

We denote by \textit{$P_{n+1}$ the path with $n+1$ vertices and $n$ edges}, that is, the graph with distinct vertices $\{p_0,...,p_n\}$ such that $p_i$ is adjacent to $p_j$ if and only if $|i-j|=1$. For a graph $H$, and a subset $X$ of $V(G)$, if $G[X]$ is a copy of $H$ in $G$, then we say that \textit{$X$ is an $H$}. By convention, when explicitly describing a path we will list the vertices in order. In this paper we are interested in understanding the class of $\{P_5,\overline{P_5}\}$-free graphs. 

Let $A$ and $B$ be disjoint subsets of $V(G)$. For a vertex $b\in V(G)\setminus A$, we say that \textit{$b$ is complete to $A$} if $b$ is adjacent to every vertex of $A$, and that \textit{$b$ is anticomplete to $A$} if $b$ is non-adjacent to every vertex of $A$. If every vertex of $A$ is complete to $B$, we say \textit{$A$ is complete to $B$}, and that \textit{$A$ is anticomplete to $B$} if every vertex of $A$ is anticomplete to $B$. If $b\in V(G)\setminus A$ is neither complete nor anticomplete to $A$, we say that \textit{$b$ is mixed on $A$}. A \textit{homogeneous set} in a graph $G$ is a subset $X$ of $V(G)$ with $1< |X|< |V(G)|$ such that no vertex of $V(G)\backslash X$ is mixed on $X$. We say that a graph is \textit{prime} if it has at least four vertices, and no homogeneous set.

Let us now define the \textit{substitution} operation. Given graphs $H_1$ and $H_2$, on disjoint vertex sets, each with at least two vertices, and $v\in V(H_1)$, we say that \textit{$H$ is obtained from $H_1$ by substituting $H_2$ for $v$}, or \textit{obtained from $H_1$ and $H_2$ by substitution} (when the details are not important) if:

\begin{itemize}
\item $V(H)=(V(H_1)\cup V(H_2))\setminus\{v\}$,
\item $H[V(H_2)]=H_2$,
\item $H[V(H_1)\setminus\{v\}]=H_1[V(H_1)\setminus \{v\}]$, and
\item $u\in V(H_1)$ is adjacent in $H$ to $w\in V(H_2)$ if and only if $w$ is adjacent to $v$ in $H_1$.
\end{itemize}

Thus, a graph $G$ is obtained from smaller graphs by substitution if and only if $G$ is not prime. Since $P_5$ and $\overline{P_5}$ are both prime, it follows that if $H_1$ and $H_2$ are $\{P_5,\overline{P_5}\}$-free graphs, then any graph obtained from $H_1$ and $H_2$ by substitution is $\{P_5,\overline{P_5}\}$-free. Hence, in this paper we are interested in understanding the class of prime $\{P_5,\overline{P_5}\}$-free graphs. 

Let $C_n$ denote the \textit{cycle of length $n$}, that is, the graph with distinct vertices $\{c_1,...,c_n\}$ such that $c_i$ is adjacent to $c_j$ if and only if $|i-j|=1$ or $n-1$. A theorem of Fouquet \cite{C5} tells us that:

\begin{lemma}\label{C5'}
Any $\{P_5,\overline{P_5}\}$-free graph that contains $C_5$ is either isomorphic to $C_5$ or has a homogeneous set.
\end{lemma} 

\noindent That is, $C_5$ is the unique prime $\{P_5$,$\overline{P_5}\}$-free graph that contains $C_5$, and so we concern ourselves with prime $\{P_5$,$\overline{P_5},C_5\}$-free graphs, the main subject of this paper. 

Let $G$ be a graph. A \textit{clique} in $G$ is a set of vertices all pairwise adjacent. A \textit{stable set} in $G$ is a set of vertices all pairwise non-adjacent. The \textit{neighborhood} of a vertex $v\in V(G)$ is the set of all vertices adjacent to $v$, and is denoted $N(v)$. A vertex $v$ is \textit{simplicial} if $N(v)$ is a clique. A vertex $v$ is \textit{antisimplicial} if $V(G)\setminus N(v)$ is a stable set, that is, if and only if $v$ is a simplicial vertex in the complement.

In \cite{MS} Hayward and Nastos proved:

\begin{lemma}\label{a}

If $G$ is a prime $\{P_5,\overline{P_5},C_5\}$-free graph, then there exists a copy of $P_4$ in $G$ whose vertices of degree one are simplicial, and whose vertices of degree two are antisimplicial.

\end{lemma}

A graph $G$ is a \textit{split graph} if there is a partition $V(G)=A\cup B$ such that $A$ is a stable set and $B$ is a clique. F\"{o}ldes and Hammer \cite{split} showed:

\begin{lemma}\label{splitt}
A graph $G$ is a split graphs if and only if $G$ is a $\{C_4,\overline{C_4},C_5\}$-free graph. 
\end{lemma}
 
 Drawn in Figure 1 with its complement, $H_6$ is the graph with vertex set $\{ v_1, v_2, v_3, v_4, v_5, v_6\}$ and edge set $\{ v_1v_2, v_2v_3, v_3v_4, v_2v_5, v_3v_6, v_5v_6 \}$.

\begin{figure}\label{A}
\begin{center}  
  \begin{tikzpicture}
   \GraphInit[vstyle=Simple]

   \tikzset{VertexStyle/.style = {shape = circle,fill = black,minimum size = 5pt,inner sep=0pt}}

     \Vertex[x=0 ,y=0, Lpos=270]{1}
     \Vertex[x=2, y=0, Lpos=225]{2} 
     \Vertex[x=4, y=0, Lpos=270]{3} 
     \Vertex[x=6, y=0, Lpos=270]{4}   

     \Vertex[x=2, y=2, Lpos=90]{5} 
     \Vertex[x=4, y=2, Lpos=90]{6}

	 \Edges(1,2,3,4)
	 \Edges(2,5,6,3)

	 \Vertex[x=8 ,y=0, Lpos=270]{7}
     \Vertex[x=10, y=0, Lpos=225]{8} 
     \Vertex[x=12, y=0, Lpos=270]{9} 

     \Vertex[x=14 ,y=0, Lpos=270]{10}
     \Vertex[x=10, y=2, Lpos=225]{11} 
     \Vertex[x=12, y=2, Lpos=270]{12} 

	 \Edges(7,8,9,10)
	 \Edges(7,11,8,12,9,11)
	 \Edges(10,12)

  \end{tikzpicture}
\end{center}  
    \caption{$H_6$ and $\overline{H_6}$.}
\end{figure}
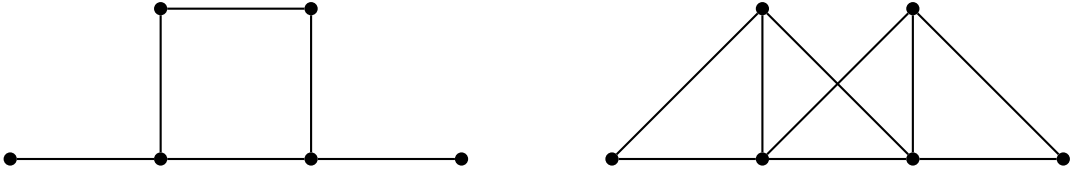

Hayward and Nastos conjectured the following:

\begin{theorem}[The $H_6$-Conjecture]\label{h6conj'}
If $G$ is a prime $\{P_5,\overline{P_5},C_5\}$-free graph which is not split, then there exists a copy of $H_6$ in $G$ or $\overline{G}$ whose two vertices of degree one are simplicial, and whose two vertices of degree three are antisimplicial.
\end{theorem}

First, in Figure 2 we provide a counterexample to \ref{h6conj'}. On the other hand, we prove the following slightly weaker version:

\begin{lemma}\label{H6SA-S'}
If $G$ is a prime $\{P_5,\overline{P_5},C_5\}$-free graph which is not split, then there exists a copy of $H_6$ in $G$ or $\overline{G}$ whose two vertices of degree one are simplicial, and at least one of whose vertices of degree three is antisimplicial.
\end{lemma}

We say that a graph $G$ \textit{admits a $1$-join}, if $V(G)$ can be partitioned into four non-empty pairwise disjoint sets $(A,B,C,D)$, where $A$ is anticomplete to $C\cup D$, and $B$ is complete to $C$ and anticomplete to $D$. In trying to use \ref{H6SA-S'} to improve upon \ref{C5'} we conjectured the following:

%Another conjecture is the following:

\begin{theorem}
\label{conj'}
If $G$ is a $\{P_5,\overline{P_5}\}$-free graph, then either

\begin{itemize}
\item $G$ is isomorphic to $C_5$, or
\item $G$ is a split graph, or
\item $G$ has a homogeneous set, or
\item $G$ or $\overline{G}$ admits a 1-join.
\end{itemize}
\end{theorem}

%Next, we give a counterexample (Figure 3) to \ref{conj'}. 
\noindent However, \ref{conj'} does not hold, and we give a counterexample in Figure 3. 

The \textit{bull} is a graph with vertex set $\{ x_1, x_2, x_3, y, z \}$ and edge set $\{ x_1x_2, x_2x_3, x_1x_3, x_1y, x_2z \}$. Lastly, applying a result of the first author and Seymour \cite{grow} we give a short proof of \ref{C5'}, and Fouquet's result \cite{C5} on the structure of $\{P_5,\overline{P_5}$,bull$\}$-free graphs.

This paper is organized as follows. Section 2 contains results about the existence of simplicial and antisimplicial vertices in $\{P_5,\overline{P_5}\}$-free graphs. In Section 3 we give a counterexample to the $H_6$-conjecture \ref{h6conj'}, and prove \ref{H6SA-S'}, a slightly weaker version of the conjecture. We also give a simpler proof of \ref{a}, and provide a counterexample to \ref{conj'}. Finally, in Section 4 we give a new proof of \ref{C5'}, and a structure theorem for $\{P_5,\overline{P_5}$,bull$\}$-free graphs.

\section{Simplicial and Antisimplicial vertices}

In this section we prove the following result:

\begin{lemma}\label{SandA} 
All prime $\{P_5,\overline{P_5},C_5\}$-free graphs have both a simplicial vertex, and an antisimplicial vertex.
\end{lemma}

Along the way we establish \ref{anti}, a result which is helpful in finding simplicial and antisimplicial vertices in prime $\{P_5,\overline{P_5}\}$-free graphs.

Let $G$ be a graph. We say $G$ is \textit{connected} if $V(G)$ cannot be partitioned into two disjoint sets anticomplete to each other. If $\overline{G}$ is connected we say that $G$ is \textit{anticonnected}. Let $X\subseteq Y\subseteq V(G)$. We say $X$ is a \textit{connected subset of $Y$} if $G[X]$ is connected, and that $X$ is an \textit{anticonnected subset of $Y$} if $G[X]$ is anticonnected. A \textit{component} of $X$ is a maximal connected subset of $X$, and an \textit{anticomponent} of $X$ is a maximal anticonnected subset of $X$. 

%A component of a non-null graph G is a maximal connected induced subgraph of G, and an anti-component of G is a maximal anti- connected induced subgraph of G.

First, we make the following three easy observations:

\begin{lemma}\label{prime}

If $G$ is a prime graph, then $G$ is connected and anticonnected.

\end{lemma}

\begin{proof}

Passing to the complement if necessary, we may suppose $G$ is not connected. Since $G$ has at least four vertices, there exists a component $C$ of $V(G)$ such that $|V(G)\setminus C|\geq 2$. However, then $V(G)\setminus C$ is a homogeneous set, a contradiction. This proves \ref{prime}.
\end{proof}

%We say a vertex \textit{$v\in V(G)\setminus X$ is mixed on an edge given by $x,y\in X$}, if $x$ is adjacent to $y$ and $v$ is mixed on $\{x,y\}$. Similarly, a vertex \textit{$v\in V(G)\setminus X$ is mixed on a non-edge given by $x,y\in X$}, if $x$ is non-adjacent to $y$ and $v$ is mixed on $\{x,y\}$.

We say a vertex \textit{$v\in V(G)\setminus X$ is mixed on an edge of $X$}, if there exist adjacent $x,y\in X$ such that $v$ is mixed on $\{x,y\}$. Similarly, a vertex \textit{$v\in V(G)\setminus X$ is mixed on a non-edge of $X$}, if there exist non-adjacent $x,y\in X$ such that $v$ is mixed on $\{x,y\}$.

\begin{lemma}\label{mixed1}

Let $G$ be a graph, $X\subseteq V(G)$, and suppose $v\in V(G)\setminus X$ is mixed on $X$.

\begin{enumerate}

\item If $X$ is a connected subset of $V(G)$, then $v$ is mixed on an edge of $X$. %given by $x,y\in X$. 

\item If $X$ is an anticonnected subset of $V(G)$, then $v$ is mixed on a non-edge of $X$.%given by $x,y\in X$.

\end{enumerate}

\end{lemma}

\begin{proof}

Suppose $X$ is a connected subset of $V(G)$. Since $v$ is mixed on $X$, both $X\cap N(v)$ and $X\setminus N(v)$ are non-empty. As $G[X]$ is connected, there exists an edge given by $x\in X\cap N(v)$ and $y\in X\setminus N(v)$, and $v$ is mixed on $\{x,y\}$. This proves \ref{mixed1}.1. Passing to the complement, we get \ref{mixed1}.2.

\end{proof}

\begin{lemma}\label{mixed2}

Let $G$ be a graph, $X_1,X_2\subseteq V(G)$ with $X_1\cap X_2=\emptyset$, and $v\in V(G)\setminus (X_1\cup X_2)$.

\begin{enumerate}

\item If $G$ is $P_5$-free, and $X_1,X_2$ are connected subsets of $V(G)$ anticomplete to each other, then $v$ is not mixed on both $X_1$ and $X_2$.

\item If $G$ is $\overline{P_5}$-free, and $X_1,X_2$ are anticonnected subsets of $V(G)$ complete to each other, then $v$ is not mixed on both $X_1$ and $X_2$.

\end{enumerate}

\end{lemma}

\begin{proof}

Suppose $G$ is $P_5$-free, $X_1,X_2$ are disjoint connected subsets of $V(G)$ anticomplete to each other, and $v$ is mixed on both $X_1$ and $X_2$. By \ref{mixed1}.1, $v$ is mixed on an edge of $X_1$, given by say $x_1,y_1\in X_1$ with $v$ adjacent to $x_1$ and non-adjacent to $y_1$, and an edge of $X_2$, given by say $x_2,y_2\in X_2$ with $v$ adjacent to $x_2$ and non-adjacent to $y_2$. However, then $\{y_1,x_1,v,x_2,y_2\}$ is a $P_5$, a contradiction. This proves \ref{mixed2}.1. Passing to the complement, we get \ref{mixed2}.2.

\end{proof}

As a consequence of \ref{mixed1} and \ref{mixed2} we obtain the following two useful results:

\begin{lemma}\label{mixed3}

Let $u$ and $v$ be non-adjacent vertices in a $\overline{P_5}$-free graph $G$, and let $A$ be an anticonnected subset of $N(u)\cap N(v)$. Then no vertex $w\in V(G)\setminus (A\cup \{u,v\})$ can be mixed on both $A$ and $\{ u,v\}$.

\end{lemma}

\begin{proof}

Since $A$ and $\{u,v\}$ are disjoint anticonnected subsets of $V(G)$ complete to each other, \ref{mixed3} follows from \ref{mixed2}.2.

\end{proof}

\begin{lemma}\label{mixed4}

Let $u,v$ and $w$ be three pairwise non-adjacent vertices in a $\{P_5, \overline{P_5}\}$-free graph $G$ such that $w$ is mixed on an anticonnected subset $A$ of $N(u)\cap N(v)$. Then no vertex $z\in N(w)\setminus (A\cup \{u,v\})$ can be mixed on $\{ u,v\}$.

\end{lemma}

\begin{proof}
Suppose there exists a vertex $z\in N(w)\setminus (A\cup \{u,v\})$ which is mixed on $\{ u,v\}$, with say $z$ adjacent to $v$ and non-adjacent to $u$. Since $w$ is mixed on $A$, by \ref{mixed1}.2, it follows that $w$ is mixed on a non-edge of $A$, given by say $x,y\in A$ with $w$ adjacent to $x$ and non-adjacent to $y$. By \ref{mixed3}, $z$ is not mixed on $A$. However, if $z$ is anticomplete to $A$, then $\{y,u,x,w,z\}$ is a $P_5$, and if $z$ is complete to $A$, then $\{x,y,w,u,z\}$ is a $\overline{P_5}$, in both cases a contradiction. This proves \ref{mixed4}.

\end{proof}

Now, we can start to prove \ref{SandA}.

\begin{lemma}\label{join}

Let $G$ be a prime $\{P_5,\overline{P_5},C_5\}$-free graph. Then $G$ has an antisimplicial vertex, or admits a 1-join.

\end{lemma}

\begin{proof}

Suppose $G$ does not admit a 1-join. Let $W$ be a maximal subset of vertices that has a partition $A_1\cup ...\cup A_k$ with $k\geq2$ such that:

\begin{itemize} 

\item $A_1,...,A_k$ are all anticonnected subsets of $V(G)$, and

\item $A_1,...,A_k$ are pairwise complete to each other.

\end{itemize}

\bigskip

\noindent\textit{(1) $V(G)\setminus W$ is non-empty.}

\bigskip

\noindent By \ref{prime}, $G$ is anticonnected, which implies that $V(G)\setminus W$ is non-empty. This proves (1).

\bigskip

\noindent\textit{(2) Every $v\in V(G)\setminus W$ is either anticomplete to or mixed on $A_i$ for each $i\in \{1,...,k\}$.}

\bigskip

\noindent Suppose $v\in V(G)\setminus W$ is complete to some $A_i$. Take $B$ to be the union of all the $A_j$ to which $v$ is complete. However, since $\{v\}\cup W\backslash B$ is anticonnected and complete to $B$, it follows that $W'=B\cup(\{v\}\cup W\backslash B)$ contradicts the maximality of $W$. This prove (2).

\bigskip

\noindent\textit{(3) If for some $i\in \{1,...,k\}$, $v\in V(G)\setminus W$ is mixed on $A_i$, then $v$ is anticomplete to $W\backslash A_i$.}

\bigskip

\noindent By \ref{mixed2}.2, any $v\in V(G)\setminus W$ is mixed on at most one $A_i$, and so together with (2) this proves (3).

\bigskip

\noindent\textit{(4) Every vertex in $V(G)\setminus W$ is mixed on exactly one $A_i$, for some $i\in \{1,...,k\}$.}

\bigskip

\noindent Suppose not. Let $X\subseteq V(G)\setminus W$ be the set of vertices anticomplete to $W$, which is non-empty by (2) and (3). By \ref{prime}, $G$ is connected, and so there exists an edge given by $v\in X$ and $u\in V(G)\setminus (X\cup W)$. By (2), $u$ is mixed on some $A_i$, and so, by \ref{mixed1}.2, $u$ is mixed on a non-edge of $A_i$, given by say $x_i,y_i\in A_i$ with $u$ adjacent to $x_i$ and non-adjacent to $y_i$. However, by (3), $u$ is anticomplete to $W\setminus A_i$, and so for $j\ne i$ and a vertex $z\in A_j$ we get that $\{v,u,x_i,z,y_i\}$ is a $P_5$, a contradiction. This proves (4).

\bigskip

\noindent And so, by (3) and (4), we can partition $V(G)=A_1\cup ...\cup A_k\cup B_1\cup ...\cup B_k$, where each $B_i$ is the set of vertices mixed on $A_i$ and anticomplete to $(A_1\cup...\cup A_k)\setminus A_i$.

\bigskip

\noindent\textit{(5) $B_1,...,B_k$ are pairwise anticomplete.}

\bigskip

\noindent Suppose for $i\ne j$, $b_i \in B_i$ is adjacent to $b_j \in B_j$. By \ref{mixed1}.2, $b_i$ is mixed on a non-edge of $A_i$, given by say $x_i,y_i\in A_i$ with $b_i$ adjacent to $x_i$ and non-adjacent to $y_i$. As $b_j$ is mixed on $A_j$, there exists $x_j\in A_j$ non-adjacent to $b_j$, however then $\{b_j,b_i,x_i,x_j,y_i\}$ is a $P_5$, a contradiction. This proves (5).

\bigskip

\noindent\textit{(6) Exactly one $B_i$ is non-empty.}

\bigskip

\noindent By (1) and (4), at least one $B_i$ is non-empty. Suppose for $i\ne j$, $B_i$ and $B_j$ are both non-empty. Then, by (5), $A=B_i$, $B=A_i$, $C=(A_1\cup ...\cup A_k)\backslash A_i$ and $D=(B_1\cup ...\cup B_k)\backslash B_i$ is a 1-join, a contradiction. This proves (6).

\bigskip

\noindent Hence, by (6), we may assume $B_1$ is non-empty while $B_2,\dots,B_k$ are all empty.

\bigskip

\noindent\textit{(7) $k=2$ and $|A_2|=1$.}

\bigskip

\noindent Since $A_2\cup ...\cup A_k$ is not a homogeneous set, (6) implies that $k=2$ and $|A_2|=1$. This proves (7).

\bigskip

\noindent Let $a$ be the vertex in $A_2$.

\bigskip

\noindent\textit{(8) $B_1$ is a stable set.}

\bigskip

\noindent Suppose not. Then there exists a component $B$ of $B_1$ with $|B|>1$. Since $a$ is anticomplete to $B_1$, and $B$ is a component of $B_1$, as $G$ is prime, it follows that there exist $a_1\in A_1$ which is mixed on $B$. Thus, by \ref{mixed1}.1, $a_1$ is mixed on an edge of $B$, given by say $b,b'\in B$ with $a_1$ adjacent to $b$ and non-adjacent to $b'$. Next, partition $A_1=C\cup D$ with $C=A_1\cap (N(b)\setminus N(b'))$ and $D=A_1\setminus C$, where both $C$ and $D$ are non-empty, as $a_1\in C$ and $b'$ is mixed on $A_1$. Since $A_1$ is anticonnected there exists a non-edge given by $c\in C$ and $d\in D$. However, since $d\in D$, it follows that $\{d,a,c,b,b'\}$ is either a $P_5,\overline{P_5}$ or $C_5$, a contradiction. This proves (8).

\bigskip

\noindent Thus, by (8), $a$ is an antisimplicial vertex. This proves \ref{join}.

\end{proof}

Next, we observe:

\begin{lemma}\label{nbrs}

Let $u$ and $v$ be non-adjacent vertices in a prime $\overline{P_5}$-free graph $G$. Then either

\begin{itemize}

\item $N(u)\cap N(v)$ is a clique, or
\item there exists a vertex $w\in V(G)\setminus(N(u)\cup N(v)\cup \{u,v\})$ which is mixed on an anticonnected subset of $N(u)\cap N(v)$.

\end{itemize}

\end{lemma}
\begin{proof}

Suppose $N(u)\cap N(v)$ is a not clique. Then there exists an anticomponent $A$ of $N(u)\cap N(v)$ with $|A|>1$. Since $\{u,v\}$ is complete to $N(u)\cap N(v)$, and $A$ is a anticomponent of $N(u)\cap N(v)$, as $G$ is prime, it follows that there exists $w\in V(G)\setminus((N(u)\cap N(v)) \cup \{u,v\})$ which is mixed on $A$. Thus, by \ref{mixed3}, $w$ is not mixed on $\{u,v\}$, and so $w$ is anticomplete to $\{u,v\}$. This proves \ref{nbrs}.

\end{proof}

A useful consequence of~\ref{nbrs} is the following:

\begin{lemma}

Let $v$ be a vertex in a prime $\{P_5,\overline{P_5}\}$-free graph $G$.

\begin{enumerate}\label{anti}

\item If $v$ is antisimplicial, and we choose $u$ non-adjacent to $v$ such that $|N(u)\cap N(v)|$ is minimum, then $u$ is a simplicial vertex. 

\item If $v$ is simplicial, and we choose $u$ adjacent to $v$ such that $|N(u)\cup N(v)|$ is maximum, then $u$ is an antisimplicial vertex.

\end{enumerate}

\end{lemma}

\begin{proof}

Suppose $v$ is antisimplicial, we choose $u$ non-adjacent to $v$ such that $|N(u)\cap N(v)|$ is minimum, and $u$ is not simplicial. Since $v$ is antisimplicial, it follows that $N(u)\setminus N(v)$ is empty, and thus, as $u$ is not simplicial, $N(u)\cap N(v)$ is not a clique. Hence, by \ref{nbrs}, there exists some $w$, non-adjacent to both $u$ and $v$, which is mixed on an anticonnected subset of $N(u)\cap N(v)$. However, then, by our choice of $u$, there exists a vertex $z\in N(v)\backslash N(u)$ adjacent to $w$, contradicting \ref{mixed4}. This proves \ref{anti}.1. Passing to the complement, we get \ref{anti}.2.

\end{proof}

\begin{lemma}\label{AS} 
Let $G$ be a prime $\{P_5,\overline{P_5},C_5\}$-free graph. Then $G$ has a simplicial vertex, or an antisimplicial vertex.
\end{lemma}

\begin{proof}
Suppose $G$ does not have an antisimpicial vertex. Then, by~\ref{join}, it admits a 1-join $(A,B,C,D)$. 

\bigskip

\noindent\textit{(1) $A$ and $D$ are stable sets.}

\bigskip

\noindent By symmetry, it suffices to argue that $A$ is a stable set. Suppose not. Then there exists a component $A'$ of $A$ with $|A'|>1$. Since $C\cup D$ is anticomplete to $A$, and $A'$ is a component of $A$, as $G$ is prime, it follows that there exists $b\in B$ which is mixed on $A'$. Thus, by \ref{mixed1}.1, $b$ is mixed on an edge of $A'$, given by say $a,a'\in A'$ with $b$ adjacent to $a'$ and non-adjacent to $a$. By \ref{prime}, $G$ is connected, and so there exists an edge given by $c\in C$ and $d\in D$. However, then $\{a,a',b,c,d\}$ is a $P_5$, a contradiction. This proves (1).

\bigskip

\noindent Next, fix some $c\in C$, and choose a vertex $a\in A$ such that $|N(a)\cap N(c)|$ is minimum.

\bigskip

\noindent\textit{(2) $a$ is a simplicial vertex.}

\bigskip

\noindent Suppose not. Then, by (1), $N(a)\cap N(c)=N(a)\subseteq B$ is not a clique, and so, by \ref{nbrs}, there exists $w$, non-adjacent to both $a$ and $c$, which is mixed on an anticonnected subset of $N(a)\cap N(c)$. Since $B$ is complete to $C$ and anticomplete to $D$, it follows that $w$ belongs to $A$. However, then, by our choice of $a$, there exists a vertex $z \in N(c)\backslash N(a)$ adjacent to $w$, contradicting \ref{mixed4}. This proves (2).

\bigskip

\noindent This completes the proof of \ref{AS}.
\end{proof}

Putting things together we can now prove \ref{SandA}.

\begin{proof}[Proof of $\ref{SandA}$]
By~\ref{AS}, passing to the complement if necessary, there exists an antisimplicial vertex $a$. And so, by \ref{anti}.1, if we choose $s$ non-adjacent to $a$ such that $|N(a)\cap N(s)|$ is minimum, then $s$ is simplicial. This proves \ref{SandA}.
\end{proof}

\section{The $H_6$-Conjecture}

In this section we give a counterexample to the $H_6$-conjecture \ref{h6conj'}, and prove \ref{H6SA-S'}, a slightly weaker version of the conjecture. We also give a proof of \ref{a}, and provide a counterexample to \ref{conj'}. 

We begin by establishing some properties of prime graphs. Recall the following theorem of Seinsche \cite{P33}:

\begin{lemma}\label{P3free}

If $G$ is a $P_4$-free graph with at least two vertices, then $G$ is either not connected or not anticonnected.
%A graph $G$ with at least two vertices is $P_4$-free if and only if $G$ is either not connected or not anticonnected.
%Theorem 2.1 If G is a graph with at least two vertices, and no induced subgraph of G is isomorphic to the three-edge path, then either G or Gc is not connected.

\end{lemma}

Together, \ref{prime} and \ref{P3free} imply the following:

\begin{lemma}\label{p1}

Every prime graph contains $P_4$.

\end{lemma}

Next, as first shown by Ho\`{a}ng and Khouzam \cite{brittle}, we observe that:

\begin{lemma}\label{simp} Let $G$ be a prime graph.

\begin{enumerate}

\item A vertex $v\in V(G)$ is simplicial if and only if $v$ is a degree one vertex in every copy of $P_4$ in $G$ containing it.

\item A vertex $v \in V(G)$ is antisimplicial if and only if $v$ is a degree two vertex in every copy of $P_4$ in $G$ containing it.

\end{enumerate}
\end{lemma}

\begin{proof}

Both forward implications are clear. To prove the converse of \ref{simp}.1, suppose there exists a vertex $v$ which is not simplicial and yet is a degree one vertex in every copy of $P_4$ in $G$ containing it. Then there exists an anticomponent $A$ of $N(v)$ with $|A|>1$. Since $v$ is complete to $A$, and $A$ is a anticomponent of $N(v)$, as $G$ is prime, it follows that there exists $u\in V(G)\setminus (N(v)\cup \{v\})$ which is mixed on $A$. Thus, by \ref{mixed1}.2, $u$ is mixed on a non-edge of $A$, given by say $x,y\in A$ with $u$ adjacent to $x$ and non-adjacent to $y$. However, then $\{y,v,x,u\}$ is a $P_4$ with $v$ having degree two, a contradiction. This proves \ref{simp}.1.  Passing to the complement, we get \ref{simp}.2.

\end{proof}

Finally, we observe that: 

\begin{lemma}\label{Aclique} Let $G$ be a prime graph.

\begin{enumerate}

\item The set of antisimplicial vertices in $G$ is a clique.

\item The set of simplicial vertices in $G$ is a stable set.

\end{enumerate}
\end{lemma}

\begin{proof}

Suppose there exist non-adjacent antisimplicial vertices $a,a'\in V(G)$. Since $a$ is antisimplicial, it follows that $N(a')\setminus N(a)$ is empty. Similarly, $N(a)\setminus N(a')$ is also empty. However, this implies that $\{a,a'\}$ is a homogeneous set in $G$, a contradiction. This proves \ref{Aclique}.1. Passing to the complement, we get \ref{Aclique}.2. 

%Let $A$ be the set of antisimplicial vertices in $G$, and suppose $A$ is not a clique. Then there exists an anticonnected component $A'$ of $A$ with $|A'|>1$. Since $A'$ is a anticonnected component of $A$, as $G$ is prime, there exists a vertex $v\in V(G)\setminus A$ which is mixed on $A'$. Thus, by \ref{mixed1}.2, $v$ is mixed on a non-edge given by $a,a'\in A'$, with say $v$ adjacent to $a'$ and non-adjacent to $a$. However, then $V(G)\setminus N(a)$ is not a stable set, a contradiction. This proves \ref{Aclique}.1. Passing to the complement, we get \ref{Aclique}.2. %$a$ is not antisimplicial, a contradiction. This proves \ref{Aclique}.1. Passing to the complement, we get \ref{Aclique}.2. 
\end{proof}

\begin{lemma}\label{split}

Let $G$ be a prime $\{P_5,\overline{P_5},C_5\}$-free graph. Let $A$ be the set of antisimplicial vertices in $G$, and let $S$ be the set of simplicial vertices in $G$. Then $G[A\cup S]$ is a split graph which is both connected and anticonnected.
\end{lemma}

\begin{proof}
\ref{Aclique} implies that $G[A\cup S]$ is a split graph, where $A$ is a clique and $S$ is a stable set. By \ref{anti}.1, every vertex in $A$ has a non-neighbor in $S$, and, by \ref{anti}.2, every vertex in $S$ has a neighbor in $A$. Thus, $G[A\cup S]$ is both connected and anticonnected. This proves \ref{split}.
 
\end{proof}

We are finally ready to give a proof of \ref{a}, first shown in \cite{MS} by Hayward and Nastos.

\begin{lemma}\label{SAAS}

If $G$ is a prime $\{P_5,\overline{P_5},C_5\}$-free graph, then there exists a copy of $P_4$ in $G$ whose vertices of degree one are simplicial, and whose vertices of degree two are antisimplicial.
\end{lemma}

\begin{proof}

Let $A$ be the set of antisimplicial vertices in $G$, and let $S$ be the set of simplicial vertices in $G$. By \ref{SandA}, both $A$ and $S$ are non-empty. Hence, $G[A\cup S]$ is a graph with at least two vertices, which, by \ref{split}, is both connected and anticonnected, and so, by \ref{P3free}, it follows that $G[A\cup S]$ contains $P_4$. Since \ref{Aclique} implies that $A$ is a clique and $S$ is a stable set, it follows that every copy of $P_4$ in $G[A\cup S]$ is of the desired form. This proves \ref{SAAS}.

\end{proof}

Next, we turn our attention to the $H_6$-conjecture. A result of Ho\`{a}ng and Reed \cite{C4} implies the following:

\begin{theorem}\label{b}

If $G$ is a prime $\{P_5,\overline{P_5},C_5\}$-free graph which is not split, then $G$ or $\overline{G}$ contains $H_6$.  

\end{theorem}

In hopes of saying more along these lines, motivated by \ref{SAAS} and \ref{b}, Hayward and Nastos posed \ref{h6conj'}, which we restate:

\begin{theorem}[The $H_6$-Conjecture]\label{h6conj}
If $G$ is a prime $\{P_5,\overline{P_5},C_5\}$-free graph which is not split, then there exists a copy of $H_6$ in $G$ or $\overline{G}$ whose two vertices of degree one are simplicial, and whose two vertices of degree three are antisimplicial.
\end{theorem}

\begin{figure}\label{B}
\begin{center}
  \begin{tikzpicture}
   \GraphInit[vstyle=Classic]

   \tikzset{VertexStyle/.style = {shape = circle,fill = black,minimum size = 5pt,inner sep=0pt}}

     \Vertex[x=0 ,y=0, Lpos=270]{1}
     \Vertex[x=2, y=0, Lpos=270]{2} 
     \Vertex[x=4, y=0, Lpos=270]{3} 
     \Vertex[x=6, y=0, Lpos=270]{4}

	 \Vertex[x=0 ,y=2, Lpos=180]{8}
     \Vertex[x=2, y=2, Lpos=90]{7} 
     \Vertex[x=4, y=2, Lpos=90]{6} 
     \Vertex[x=6, y=2]{5}

     \Vertex[x=0 ,y=4, Lpos=90]{12}
     \Vertex[x=2, y=4, Lpos=90]{11} 
     \Vertex[x=4, y=4, Lpos=90]{10} 
     \Vertex[x=6, y=4, Lpos=90]{9}

	 \Edges(1,2,3,4)
     \Edges(8,7,6,5)
     \Edges(12,11,10,9)
     \Edges(12,8,11,6,9,5,10,7,12)
     \Edges(1,7,3,5)
     \Edges(8,2,6,4)
     \Edges(2,7)
     \Edges(3,6)
  \end{tikzpicture}
  \caption{Counterexample to the $H_6$-conjecture, where additionally $\{ 2,3\}$ is complete to $\{ 9,10,11,12\}$.}
\end{center}
\end{figure}
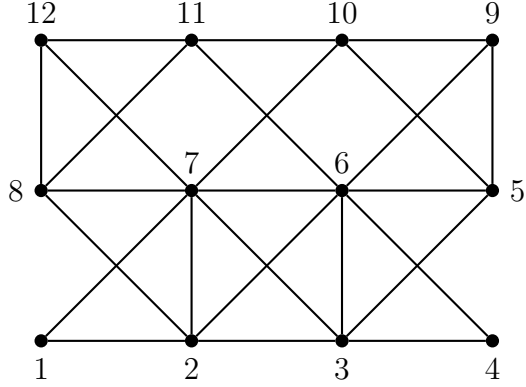

In Figure 2 we give a counterexample to \ref{h6conj}. The graph $G$ in Figure 2 contains $C_4$, and so, by \ref{splitt}, is not split. The mapping $\phi :V(G)\rightarrow V(\overline{G})$

$$\phi:=\left(\begin{array}{cccccccccccc}1 & 2 & 3 & 4 & 5 & 6 & 7 & 8 & 9 & 10 & 11 & 12 \\3 & 1 & 4 & 2 & 7 & 5 & 8 & 6 & 11 & 9 & 12 & 10\end{array}\right)$$

\noindent is an isomorphism between $G$ and $\overline{G}$. Thus, as $G$ is self-complementary, it suffices to check that $G$ is $P_5$-free, which is straight forward, as is verifying that $G$ is prime, and we leave the details to the reader. The set of simplicial vertices in $G$ is $\{ 1,4\}$, and the set of antisimplicial vertices in $G$ is $\{ 2,3\}$. However, no copy of $C_4$ in $G$ contains $\{2,3\}$, and so there does not exist a copy of $H_6$ of the desired form.

However, all is not lost as we can prove \ref{H6SA-S'}, a slightly weaker version of the $H_6$-conjecture, which we restate:

\begin{lemma}\label{H6SA-S}
If $G$ is a prime $\{P_5,\overline{P_5},C_5\}$-free graph which is not split, then there exists a copy of $H_6$ in $G$ or $\overline{G}$ whose two vertices of degree one are simplicial, and at least one of whose vertices of degree three is antisimplicial.
\end{lemma}

\begin{proof} By~\ref{SAAS}, there exist simplicial vertices $s,s'$, and antisimplicial vertices $a,a'$ such that $\{s,a,a',s'\}$ is a $P_4$ in $G$. Now, choose maximal subsets $A$ of antisimplicial vertices in $G$, and $S$ of simplicial vertices in $G$ such that $a,a'\in A$, $s,s'\in S$, every vertex in $A$ has a neighbor in $S$, and every vertex in $S$ has a non-neighbor in $A$. 

\bigskip

\noindent\textit{(1) Any graph containing a vertex which is both simplicial and antisimplicial is split.}

\bigskip

\noindent By definition, if a vertex $v\in V(G)$ is both simplicial and antisimplicial, then $N(v)$ is a clique and $V(G)\setminus N(v)$ is a stable set. This proves (1).

\bigskip

\noindent\textit{(2) There exists no vertex $v\in V(G)\setminus (A\cup S)$ adjacent to a vertex $u\in S$ and non-adjacent to a vertex $w\in A$.}

\bigskip

\noindent Suppose not. If $u$ is adjacent to $w$, then $N(u)$ is not a clique, and if $u$ is non-adjacent to $w$, then $V(G)\setminus N(w)$ is not a stable set, in both cases a contradiction. This proves (2).

\bigskip

\noindent By (1) and (2), we can partition $V(G)=A\cup S\cup B\cup C\cup D$, where $B$ is the set of vertices complete to $A$ and anticomplete to $S$, $C$ is the set of vertices complete to $A$ with a neighbor in $S$, and $D$ is the set of vertices anticomplete to $S$ with a non-neighbor in $A$. Recall \ref{Aclique} implies that $A$ is a clique and $S$ is a stable set.
 
\bigskip

\noindent\textit{(3) No vertex of $C\cup D$ is simplicial or antisimplicial.}

\bigskip

\noindent Consider a vertex $c\in C$. Then there exists $s_c\in S$ adjacent to $c$. Hence, $c$ is not antisimplicial, as otherwise we could add $c$ to $A$ contrary to maximality. By construction, $s_c$ has a non-neighbor $a_c\in A$. Since $c$ is complete to the $A$, it follows that $N(c)$ is not a clique, and thus $c$ is not simplicial. Hence, $C$ contains no simplicial or antisimplicial vertices. Passing to the complement, we get that no vertex in $D$ is simplicial or antisimplicial. This proves (3).

\bigskip

\noindent\textit{(4) We may assume that $C$ is a clique, and $D$ is a stable set.}

\bigskip

\noindent By symmetry, it is enough to argue that if $D$ is not a stable set, then the theorem holds. Suppose we have an edge given by $x,y\in D$. By definition, any antisimplicial vertex is adjacent to at least one of $x$ and $y$. And so, as $x$ and $y$ both have non-neighbors in $A$, there exists $a_x,a_y\in A$ such that $a_x$ is adjacent to $x$ and non-adjacent to $y$, and $a_y$ is adjacent to $y$ and non-adjacent to $x$. Since $S$ is anticomplete to $D$, it follows that $a_x$ and $a_y$ do not have a common neighbor $s''\in S$, as otherwise $\{ a_x,y,s'',x,a_y \}$ is a $\overline{P_5}$. By construction, every vertex in $A$ has a neighbor in $S$, and so there exists $s_x\in S$ adjacent to $a_x$ and non-adjacent to $a_y$, and $s_y\in S$ adjacent to $a_y$ and non-adjacent to $a_x$. However, then $\{ s_x,a_x,a_y,s_y,x,y \}$ is a copy of $H_6$ in $G$ of the desired form. Passing to the complement, we may also assume that $C$ is a clique. This proves (4).

\bigskip

\noindent\textit{(5) For all $d\in D$ and $u\in A$, $N(d)\subseteq N(u)\cup \{u\}$.}

\bigskip

\noindent By (4), $A\cup C$ is a clique and $D\cup S$ is a stable set. Thus, for any $d\in D$, it follows that $N(d)\subseteq A\cup B\cup C$. Since $A$ is complete to $B$, it follows that any $a\in A$ is complete to $(A\setminus \{a\})\cup B \cup C$. This proves (5).

\bigskip

\noindent\textit{(6) We may assume both $C$ and $D$ are empty.}

\bigskip

\noindent By symmetry, it is enough to argue that if $D$ is non-empty, then the theorem holds. Suppose $D$ is non-empty, and choose $d\in D$ with $|N(d)|$ minimum. Then there exists $a_d\in A$ non-adjacent to $d$. By (3) and (5), $N(a_d)\cap N(d)=N(d)$ is not a clique, and so, by \ref{nbrs}, there exists a vertex $w$, non-adjacent to both $a_d$ and $d$, which is mixed on an anticonnected subset of $N(d)$. Since $a_d$ is complete to $(A\setminus \{a_d\})\cup B\cup C$, it follows that $w\in D\cup S$. If $w\in D$, then, by our choice of $d$, there exists $z\in N(w)\setminus N(d)$ which, by (5), is adjacent to $a_d$, contradicting \ref{mixed4}. Hence, $w\in S$. Since $w$ is mixed on an anticonnected subset of $N(d)$, by \ref{mixed1}.2, $w$ is mixed on a non-edge of $N(d)$, given by say $x,y\in N(d)$ with $w$ adjacent to $x$ and non-adjacent to $y$. Since $A\cup C$ is a clique, and $B$ is complete to $A$ and anticomplete to $S$, it follows that $x\in C$ and $y\in B$. By construction, every vertex in $A$ has a neighbor in $S$, and so there exists $s_d \in S$ adjacent to $a_d$. Since $s_d$ is mixed on $\{a_d,d\}$ and non-adjacent to $y$, \ref{mixed3} implies that $s_d$ is anticomplete to $\{x,y\}$. However, then $\{ s_d,a_d,x,w,y,d \}$ is a copy of $H_6$ in $G$ of the desired form. Passing to the complement, we may also assume that $C$ is empty. This proves (6).

\bigskip

\noindent By (6), since $G$ is prime, it follows that $|B|\leq 1$, implying that $G$ is a split graph, a contradiction. This proves \ref{H6SA-S}.

\end{proof}

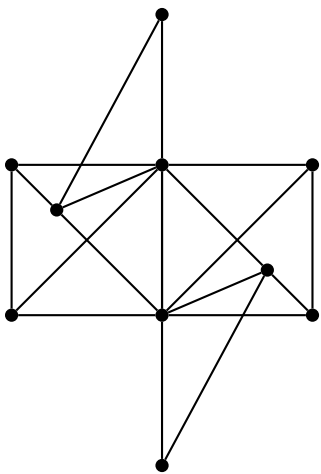
\begin{figure}\label{C}
\begin{center}
  \begin{tikzpicture}
   \GraphInit[vstyle=Simple]

   \tikzset{VertexStyle/.style = {shape = circle,fill = black,minimum size = 5pt,inner sep=0pt}}

     \Vertex[x=0 ,y=0, Lpos=270]{3}
     \Vertex[x=2, y=0, Lpos=225]{6} 
     \Vertex[x=4, y=0, Lpos=270]{4} 

	 \Vertex[x=0 ,y=2, Lpos=90]{1}
     \Vertex[x=2, y=2, Lpos=45]{5} 
     \Vertex[x=4, y=2, Lpos=90]{2} 
     \Vertex[x=2, y=4, Lpos=90]{7} %3
     \Vertex[x=2, y=-2, Lpos=270]{8} %-1

     \Vertex[x=.6, y=1.4, Lpos=90]{9} 
     \Vertex[x=3.4, y=.6, Lpos=90]{10}

	 \Edges(1,5,2,4,6,3,1)
	 \Edges(3,5,6,2)
	 \Edges(8,6,5,7)
	 \Edges(1,9,7)
	 \Edges(5,9,6)
	 \Edges(4,10,8)
	 \Edges(5,10,6)

  \end{tikzpicture}
\end{center}
  \caption{Counterexample to Conjecture~\ref{conj}.}
\end{figure}

Another conjecture which seemed plausible for a while is \ref{conj'}, which we restate:

\begin{theorem}
\label{conj}
If $G$ is a $\{P_5,\overline{P_5}\}$-free graph, then either

\begin{itemize}
\item $G$ is isomorphic to $C_5$, or
\item $G$ is a split graph, or
\item $G$ has a homogeneous set, or
\item $G$ or $\overline{G}$ admits a 1-join.
\end{itemize}
\end{theorem}

However, with Paul Seymour we found the counterexample in Figure 3. The graph in Figure 3 contains $C_4$ and $\overline{C_4}$, and so, by \ref{splitt}, is not split; we leave the rest of the details to the reader.

\section{$\{P_5,\overline{P_5},$bull$\}$-free Graphs}

In this section we give a short proof of \ref{C5'}, and of Fouquet's result \ref{bull} on the structure of $\{P_5,\overline{P_5}$,bull$\}$-free graphs. The following is joint work with Max Ehramn.

Let $O_k$ be the bipartite graph on $2k$ vertices with bipartition $(\{ a_1,\dots a_k\} ,\{ b_1,\dots b_k\})$ in which $a_i$ is adjacent to $b_j$ if and only if $i+j\geq k+1$. If a graph $G$ is isomorphic to $O_k$ for some $k$, then we call $G$ a \textit{half graph}. Note that by construction half graphs are prime. In \cite{grow} the first author and Seymour proved:

\begin{lemma}\label{growing}
Let $G$ be a graph, and let $H$ be a proper induced subgraph of $G$. Assume that both $G$ and $H$ are prime, and that both $G$ and $\overline{G}$ are not half graphs. Then there exists an induced subgraph $H'$ of $G$, isomorphic to H, and a vertex $v\in V(G)\setminus V(H')$, such that $G[V (H')\cup \{v\}]$ is prime.

\end{lemma}

Next, we give a proof of Fouquet's result \ref{C5'}, which we restate:

\begin{lemma}\label{C5two}

If $G$ is a prime $\{P_5,\overline{P_5}\}$-free graph which contains $C_5$, then $G$ is isomorphic to $C_5$.

\end{lemma}

\begin{proof}
Suppose not, and so $C_5$ is a proper induced subgraph of $G$. Since $C_5$ is self-complementary both $G$ and $\overline{G}$ contain an odd cycle, hence are non-bipartite, and thus not half graphs. As $C_5$ is prime, by \ref{growing}, there exists a subgraph $H$ induced by $\{ v_1,v_2,v_3,v_4,v_5\}$ isomorphic to $C_5$, and a vertex $v\in V(G)\backslash V(H)$ such that the subgraph of $G$ induced by $V(H)\cup\{v\}$ is prime. Considering the complement, we may assume $v$ is adjacent to at most two vertices in $V(H)$. To avoid a homogeneous set in $G[V(H)\cup\{v\}]$, by symmetry, the only possibilities are for $N(v)=\{v_1\}$, in which case $\{ v,v_1,v_2,v_3,v_4\}$ is a $P_5$, or for $N(v)=\{v_1,v_2\}$, in which case $\{ v,v_2,v_3,v_4,v_5\}$ is a $P_5$, in both cases a contradiction. This proves \ref{C5two}. 
\end{proof}

Thus, to understand prime $\{P_5,\overline{P_5}$,bull$\}$-free graphs it is enough to study prime $\{P_5,\overline{P_5},C_5$,bull$\}$-free graphs.

\begin{lemma}\label{obst}

If $G$ is a prime $\{P_5,\overline{P_5}, C_5,$bull$\}$-free graph, then either $G$ or $\overline{G}$ is a half graph.

\end{lemma}

\begin{proof}
Suppose not. By \ref{p1}, $G$ contains $P_4$, which is isomorphic to $O_2$. Since $G$ and $\overline{G}$ are not half graphs, it follows that $P_4$ is a proper induced subgraph of $G$. As $P_4$ is prime, by \ref{growing}, there exists a subgraph $H$ induced by $\{ v_1,v_2,v_3,v_4\}$ isomorphic to $P_4$, and a vertex $v\in V(G)\backslash V(H)$ such that the subgraph of $G$ induced by $V(H)\cup\{v\}$ is prime. Considering the complement, we may assume $v$ is adjacent to at most two vertices in $H$. To avoid a homogeneous set in $G[V(H)\cup\{v\}]$, by symmetry, the only possibilities are for $N(v)=\{v_1\}$, in which case $\{ v,v_1,v_2,v_3,v_4\}$ is a $P_5$, for $N(v)=\{v_1,v_4\}$, in which case $\{ v,v_1,v_2,v_3,v_4\}$ is a $C_5$, or for $N(v)=\{v_2,v_3\}$, in which case $\{ v,v_1,v_2,v_3,v_4\}$ is a bull, in all cases a contradiction. This proves \ref{obst}.
\end{proof}

Putting things together we obtain Fouquet's original structural result \cite{C5}:

\begin{theorem}\label{bull}
If $G$ is a $\{P_5,\overline{P_5}$,bull$\}$-free graph, then either

\begin{itemize}
\item $|V(G)|\leq 2$, or
\item $G$ is isomorphic to $C_5$, or
\item $G$ has a homogeneous set, or
\item $G$ or $\overline{G}$ is a half graph.
\end{itemize}
\end{theorem}

\begin{proof}

As all graphs on three vertices have a homogeneous set, \ref{bull} immediately follows from \ref{C5two} and \ref{obst}.

\end{proof}

\section{Acknowledgement}

We would like to thank Ryan Hayward, James Nastos, Irena Penev, Matthieu Plumettaz, Paul Seymour, and Yori Zwols for useful discussions, and Max Ehramn for telling us about the $H_6$-conjecture.

\end{document}